\documentclass[12pt,a4paper]{amsart}

\usepackage{amsmath, amssymb, amscd}

\newcommand{\N}{\mathbb N}
\newcommand{\Z}{\mathbb Z}
\def\P{\mathbb P}
\newcommand{\Q}{\mathbb Q}

\def\bad{\text{\rm bad}}
\def\Int{\text{\rm Int}\,}
\def\lcm{\text{\rm lcm}}
\def\mc{\mathcal}
\def\phi{\varphi}
\def\res#1{\text{\rm Res}_{#1}\,}

\newtheorem {Th}{Theorem}
\newtheorem {Cor}[Th]{Corollary}
\newtheorem {Le}[Th]{Lemma}
\newtheorem {Pro}[Th]{Proposition}

\theoremstyle{remark}
\newtheorem*{Rem}{Remarks}
\newtheorem {Ex}{Example}

\begin{document}

\title [Polynomial parametrization of the solutions of Diophantine equations]
{Polynomial parametrization of the solutions of Diophantine equations of genus 0}

\author{Sophie Frisch \ and \ G\"unter Lettl}

\dedicatory{Dedicated to Prof.~W{\l}adys{\l}aw Narkiewicz on the occasion of his 70$^{\text th}$ birthday.}

\address{Institut f\"ur Mathematik A\\
Technische Universit\"at Graz\\
Steyrergasse 30\\ A-8010  Graz, AUSTRIA}
\email{frisch@blah.math.tu-graz.ac.at}

\address{Institut f\"ur Mathematik und wissenschaftliches Rechnen\\
Karl-Franzens-Universit\"at\\
Heinrichstra{\ss}e 36\\ A-8010  Graz, AUSTRIA}
\email{guenter.lettl@uni-graz.at}


\begin{abstract}
Let $f \in \Z[X,Y,Z]$ be a non-constant, absolutely irreducible, homogeneous polynomial with integer coefficients, such that the projective curve given by $f=0$ has a function field isomorphic to the rational function field $\Q (T)$. We show that all integral solutions of the Diophantine equation $f=0$ (up to those corresponding to some singular points) can be parametrized by a single triple of integer-valued polynomials. In general, it is not possible to parametrize this set of solutions by a single triple of polynomials with integer coefficients. \bigskip
\end{abstract}

\maketitle

Recently, the first author and L.~Vaserstein proved that the set of all Pythagorean triples can be parametrized by a single triple of integer-valued polynomials, but not by a single triple of polynomials with integer coefficients (in any number of variables) \cite{FrV}. We denote by $\Int(\Z^m)$ the ring of integer-valued polynomials in $m$ variables,
\[ \Int(\Z^m) = \{ \phi \in \Q[X_1, \dots, X_m] \mid \phi(\Z^m) \subset \Z \}\,.
\]
In this paper we will generalize the affirmative part of \cite{FrV} to such homogeneous equations as define a (plane) projective curve with a rational function field.
\smallskip

Throughout this paper, $f \in \Z[X,Y,Z] \setminus \{0\}$ denotes an irreducible polynomial with integer coefficients, which is homogeneous of degree $n \ge 1$. Let $\overline{\Q}$ be an algebraic closure of $\Q$ and $C_f \subset \P^2 (\overline{\Q})$ the plane projective curve defined by $f=0$,
\[ C_f = \bigl\{ (x:y:z) \in \P^2 (\overline{\Q}) \mid f(x,y,z)=0 \bigr\}\ .
\]
We will further suppose that the function field $K = \Q (C_f)$ of $C_f$ over $\Q$ is isomorphic to the rational function field $\Q(T)$. This implies that $f$ is absolutely irreducible (i.e., irreducible in $\overline{\Q}[X,Y,Z]$). Our assumption is satisfied, for instance, if $C_f$ has genus $0$ and possesses a regular point defined over $\Q$.

Recall that a point $(x:y:z) \in C_f$ is singular if and only if the local ring $R_{(x:y:z)} \subset K$ of all rational functions of $C_f$ that are defined at $(x:y:z)$ is not a discrete valuation ring (cf. \cite[pp. 56-57]{K}). In this case, there are finitely many discrete valuation rings $\mc O_{P_i} \subset K$ above $R_{(x:y:z)}$ (meaning $R_{(x:y:z)} \subset \mc O_{P_i}$ and $\frak m_{(x:y:z)} \subset P_i$, where $\frak m_{(x:y:z)}$ and $P_i$ denote the corresponding maximal ideals). Let $C_f^{\bad}$ denote the set of those singular points $(x:y:z) \in C_f$ for which there exists no discrete valuation ring $\mc O_P$ above $R_{(x:y:z)}$ with $\mc O_P /P \simeq \Q$. These points will be ``bad'' for our main theorem.
\smallskip

We investigate the set of integer solutions of the Diophantine equation $f(X,Y,Z) = 0$,
\[ \mc L_f := \bigl\{ (x,y,z) \in \Z^3 \mid f(x,y,z) =0 \bigr\}\,,
\]
up to those solutions which correspond to the ``bad'' points of the curve. We set
\[ \mc L_f^{\bad} = \{ (x,y,z) \in \mc L_f \mid (x:y:z) \in C_f^{\bad} \}\,.
\]

\begin{Th} \label{Th1}
Let $f \in \Z[X,Y,Z] \setminus \{0\}$ be an irreducible, homogeneous polynomial of degree $n \ge 1$ such that the function field $K = \Q (C_f)$ is isomorphic to $\Q(T)$.\newline
Then there exist polynomials $g_1, g_2, g_3 \in \Int(\Z^m)$ for some $m \in \N$ such that
\[ \mc L_f \setminus \mc L_f^{\bad} = \Bigl\{ \bigl( g_1(\underline{x}), g_2(\underline{x}),  g_3(\underline{x}) \bigr) \ \Big| \ \underline{x} \in \Z^m \Bigr\}\, ;
\]
in other words, up to the ``bad'' solutions, all solutions of the Diophantine equation
\begin{equation}\label{eq1} f(X,Y,Z) = 0
\end{equation}
can be parametrized by one triple of integer-valued polynomials.
\end{Th}
\medskip

The suppositions of Theorem \ref{Th1} imply that for $n \le 2$ the curve $C_f$ has no singular point. For $n=1$, $C_f$ is just a line and the result of Theorem \ref{Th1} is obvious (even with $g_i \in \Z[U,V]$). For $n=2$, we immediately obtain

\begin{Cor} \label{C1}
Let $f \in \Z[X,Y,Z]$ be an absolutely irreducible quadratic form. Then there exist polynomials $g_1, g_2, g_3 \in \Int(\Z^m)$ for some $m \in \N$ such that
\[ \mc L_f = \Bigl\{ \bigl( g_1(\underline{x}), g_2(\underline{x}), g_3(\underline{x}) \bigr) \ \Big|
\ \underline{x} \in \Z^m \Bigr\}\ .
\]
\end{Cor}
\medskip

For the proof of Theorem \ref{Th1} we will use the resultant of polynomials and therefore recall some well-known results on it (cf. \cite[Chap. I, \S 9-10]{W}).\newline Given polynomials $g,h \in \Z[U,V]$ in the variables $U,V$, let $\res{V} (g,h) \in \Z[U]$ denote the resultant of $g,h$ when considered as polynomials in the variable $V$ over the ring $\Z[U]$, and, vice versa, $\res{U} (g,h) \in \Z[V]$ the resultant of $g,h$ as polynomials in $U$.

\begin{Le} \label{L1} Let $g,h \in \Z[U,V]$ be relatively prime polynomials.
\begin{itemize}
\item[\textbf{a)}] Then $\res{U} (g,h) \ne 0$ and $\res{V} (g,h) \ne 0$, and there exist polynomials $r, s, r', s' \in \Z[U,V]$ with
\[ gr + hs = \res{U} (g,h) \quad \text{and} \quad gr' + hs' = \res{V} (g,h)\,.
\]
\item[\textbf{b)}] If $g$ and $h$ are homogeneous of degree $d_1$ and $d_2$, resp., then $\res{U} (g,h)$ and $\res{V} (g,h)$ are each homogeneous of degree $d_1d_2$, and consequently
\[ \res{U} (g,h) = a\, V^{d_1d_2} \quad \text{and} \quad \res{V} (g,h) = b\, U^{d_1d_2} \quad \text{with} \quad a,b \in \Z \setminus \{0\}\,.
\]
\end{itemize}
\end{Le}
\medskip

We will also use the implication (D)$\Rightarrow$(B) of the main theorem of \cite{Fr}, which for the sake of completeness we state in the following

\begin{Pro} \label{P1} Let $k \in \N$ and suppose that $S \subset \Z^k$ is the set of integer $k$-tuples in the range of a $k$-tuple of polynomials with rational coefficients, as the variables range through the integers, i.e., there exist $h_1, \dots, h_k \in \Q[X_1, \dots, X_r]$ for some  $r \in \N$ such that
\[S=\{ (h_1(\underline{x}),\dots, h_k(\underline{x})) \mid \underline{x}\in\Z^r \} \cap \Z^k\ .
\]
Then $S$ is parametrizable by a $k$-tuple of integer-valued polynomials, i.e., there exist $g_1, \dots, g_k \in \Int(\Z^m)$ for some $m \in \N$ such that
\[S= \{ (g_1(\underline{x}), \dots, g_k(\underline{x})) \mid \underline{x} \in \Z^m \}\ .
\]
\end{Pro}
\medskip

\begin{proof}[Proof of Theorem \ref{Th1}]
Let $f$ be as in the statement of the theorem. Then there exist homogeneous polynomials $h_1, h_2, h_3 \in \Q[U,V]$ such that
\[ (X, Y, Z) = \Bigl( h_1(U,V), h_2(U,V), h_3(U,V) \Bigr)
\]
defines a birational (projective) isomorphism between $C_f$ and the projective line. We may assume $h_1, h_2, h_3 \in \Z[U,V]$ and $\gcd(h_1, h_2, h_3) = 1$ (see, for instance, \cite[Sect. 2]{PV}).

For every $\Q$-rational point $(u:v) \in \P^1(\Q)$, $\bigl( h_1(u,v): h_2(u,v): h_3(u,v) \bigr)$ is the evaluation of the birational isomorphism at this point. This means that $\bigl( h_1(u,v) : h_2(u,v) : h_3(u,v) \bigr)$ is a $\Q$-rational point of $C_f$ and its local ring is contained in some discrete valuation ring of $K$ of degree $1$. Therefore
\begin{multline*} \mc L_{\Q} := \Bigl\{ \bigl( w\,h_1(u,v), w\,h_2(u,v), w\,h_3(u,v) \bigr) \ \Big| \ u,v,w \in \Q \Bigr\} = \\
\Bigl\{ \bigl( w\,h_1(u,v), w\,h_2(u,v), w\,h_3(u,v) \bigr) \ \Big| \ w \in \Q,\ u,v \in \Z \text{ with } \gcd (u,v)=1 \Bigr\}
\end{multline*}
is exactly the set of all rational solutions of \eqref{eq1} except for those corresponding to points of $C_f^{\bad}$, and $\mc L_f \setminus \mc L_f^{\bad} = \mc L_{\Q} \cap \Z^3$ is just the set of all integral triples of $\mc L_{\Q}$.

We claim that there exists some $d \in \N$ such that for all $u, v \in \Z$ with $\gcd (u,v)=1$ it follows that
\[ \gcd \bigl( h_1(u,v), h_2(u,v), h_3(u,v) \bigr) \bigm | d\ .
\]

\noindent Let $\gcd (h_1, h_2) = t \in \Z[U,V]$ and put $h_i = t\,h_i'$ with $h_i' \in \Z[U,V]$, $i=1,2$. Since $h_1', h_2'$ are relatively prime, we obtain that $\res{V} (h_1', h_2') = a\, U^\delta$ with some $0 \ne a \in \Z$ and $\delta \ge 0$, and polynomials $\rho_1, \rho_2 \in \Z[U,V]$ with $\rho_1 h_1 + \rho_2 h_2 = at U^\delta$. Since $h_1$, $h_2$, $h_3$ were assumed to be relatively prime, $\gcd ( atU^\delta, h_3) = cU^\alpha$ with $c \in \Z$ and $0 \le \alpha \le \delta$. Dividing both $atU^\delta$ and $h_3$ by $cU^\alpha$ and applying the same reasoning as above we finally obtain that there are $0 \ne a_1 \in \Z$, $\delta_1 \ge 0$ and polynomials $\phi_1, \phi_2, \phi_3 \in \Z[U,V]$ with
\begin{equation}\label{eq2}
\phi_1 h_1 + \phi_2 h_2 + \phi_3 h_3 = a_1 U^{\delta_1}\ .
\end{equation}

\noindent Using $\res{U}$ in the same way, we obtain polynomials $\psi_1, \psi_2, \psi_3 \in \Z[U,V]$, $0 \ne a_2 \in \Z$ and $\delta_2 \ge 0$ such that
\begin{equation}\label{eq3}
\psi_1 h_1 + \psi_2 h_2 + \psi_3 h_3 = a_2 V^{\delta_2}\ .
\end{equation}

\noindent For any $u,v \in \Z$ with $\gcd (u,v) =1$, \eqref{eq2} and \eqref{eq3} imply that $\gcd \bigl( h_1(u,v), h_2(u,v), h_3(u,v) \bigr)$ divides both $a_1 u^{\delta_1}$ and $a_2 v^{\delta_2}$. It follows that
\[ \gcd \bigl( h_1(u,v), h_2(u,v), h_3(u,v) \bigr) \bigm | \lcm (a_1, a_2) := d\ .
\]

\noindent So we obtain polynomials $k_i = \frac 1 d h_i \in \Q[U,V]$ with rational coefficients such that
\[ \mc L_f \setminus \mc L_f^{\bad} = \Bigl\{ \bigl( w\,k_1(u,v), w\,k_2(u,v), w\,k_3(u,v) \bigr) \ \Big| \ u,v,w \in \Z \Bigr\} \cap \Z^3\ .
\]

\noindent Now we apply Proposition \ref{P1}, which yields the assertion of Theorem \ref{Th1}.
\end{proof}
\pagebreak

\begin{Rem}
If the integers $a_1, a_2$ appearing in \eqref{eq2} and \eqref{eq3} in the proof of Theorem \ref{Th1} are both equal to $1$, then $k_i = h_i \in \Z[U,V]$ and $\mc L_f \setminus \mc L_f^{\bad}$ can actually be parametrized by a triple of polynomials with integral coefficients (compare Example 2 below).\newline
When applying Proposition \ref{P1}, we have no information about the number $m$ of variables of the integer-valued polynomials $g_i$ appearing in Theorem \ref{Th1}.
\end{Rem}
\medskip

\begin{Ex} This example shows that for $n \ge 3$ ``bad'' singular points may appear.\break Consider
\[ f = X^3 + Y^3 + X^2Z - 2Y^2Z \in \Z[X,Y,Z]\,.
\]
Then $(0:0:1) \in C_f$ is a singular point. Only one discrete valuation ring lies over the local ring $R_{(0:0:1)}$, and this valuation ring has residue class field isomorphic to $\Q(\sqrt 2)$.\break A birational (projective) isomorphism between $C_f$ and the projective line is given by
\[ (X:Y:Z) = \Bigl( (V(2U^2-V^2)) : (U(2U^2-V^2)) : (V^3+U^3) \Bigr)\,,
\]
but there is no $\Q$-rational point $(u:v) \in \P^1(\Q)$ corrsponding to the singular point $(0 : 0 : 1)$. Indeed, the corresponding point $(u:v) = (1 : \sqrt 2)$ is only defined over $\Q(\sqrt 2)$.
\end{Ex}
\medskip

\begin{Ex} In contrast to the Pythagorean triples (corresponding to the unit circle, see \cite{FrV}), we know that for the equilateral hyperbola the set $\mc L_f$ can be parametrized by a single triple of polynomials with integer coefficients. Let
\[ f = XY - Z^2 \in \Z[X,Y,Z]\,.
\]
All $\Q$-rational points of $C_f$ are given by $(u^2 : v^2 : uv)$ with $(u :v) \in \P^1 (\Q)$. If $u,v \in \Z$ with $\gcd(u,v)=1$ then also $\gcd(u^2, v^2, uv) =1$. So the set of all integral solutions of $XY - Z^2 = 0$ is given by
\[ \{ (u^2w, v^2w, uvw) \mid u,v,w \in \Z \}\,.
\]
\end{Ex}

\end{document}